\documentclass[11pt]{amsart}
\usepackage{graphicx}
\usepackage{amssymb}
\usepackage{amsmath}
\usepackage{amsthm}
\usepackage{mathtools}
\usepackage{tikz}
\usepackage{psfrag}
\usepackage{xcolor}

\newtheorem{thm}{Theorem}[section]
\newtheorem{lem}[thm]{Lemma}
\newtheorem{pro}[thm]{Proposition}

\theoremstyle{definition}

\newtheorem{remk}{Remark}[section]

\newcommand{\M}{\mathbf M}
\newcommand{\MN}{\mathbf M^{(1+n)\times (n+1)}}

\newcommand{\B}{\mathbf B}
\newcommand{\R}{\mathbf R}

\newcommand{\tr}{\operatorname{tr}}

\newcommand{\dv}{\operatorname{div}}
\newcommand{\Dv}{\operatorname{Div}}
\newcommand{\Curl}{\operatorname{Curl}}

\newcommand{\wcon}{\rightharpoonup}

\makeatletter
\numberwithin{equation}{section}
\makeatother

\begin{document}
  \author{Menglan Liao}
 \address{School of Mathematics\\  Jilin University\\ Changchun, Jilin Province 130012, China 
 \& 
 Department of Mathematics\\ Michigan State University\\ East Lansing, MI 48824, USA}
      \email{liaomen1@msu.edu}
      
 \author{Lianzhang Bao}
          \address{School of Mathematics\\ Jilin University\\ Changchun, Jilin Province 130012, China
          \& School of Mathematical Science\\ Zhejiang University\\ Hangzhou, Zhejiang Province 310027, China}
    \email{lzbao@jlu.edu.cn}

   \author{Baisheng Yan}
           \address{Department of Mathematics\\ Michigan State University\\ East Lansing, MI 48824, USA}
     \email{yanb@msu.edu}

\title{On weak closure of some diffusion equations}

\subjclass[2010]{Primary 35Q99,  35B35. Secondary 49J45}
\keywords{diffusion equation, approximating sequence, weak closure, compensated  compactness, quasiconvexity}

\begin{abstract}
  We study the closure  of approximating sequences of some diffusion equations under certain weak convergence.  A specific description of the closure under weak $H^1$-convergence is given, which reduces to the original equation when the equation is parabolic. However, the closure under strong $L^2$-convergence may be much larger, even for  parabolic equations.
\end{abstract}

\maketitle

\section{Introduction}
Let $\Omega$ be a bounded domain  in $\R^n$ with Lipschitz boundary,  $0<T<\infty$  a given number, and  $\Omega_T=\Omega\times  (0,T).$  We study the  diffusion equation
\begin{equation}\label{eq0}
u_t=\dv\sigma(Du) \quad \mbox{on $\; \Omega_T,$}
\end{equation}
where  $\sigma\colon \R^n\to \R^n$ is a given continuous function representing the diffusion flux, $u=u(x,t)$ is a unknown scalar function, and $Du$ and $u_t$ are the spatial gradient  and time derivative of $u$, respectively. 

If  $\sigma$ is {\em monotone}; that is, $(\sigma(q)-\sigma(p))\cdot (q-p)\ge 0$ for all $p,q\in\R^n$, then   the equation (\ref{eq0}) is called {\em parabolic}. For parabolic equations,  initial-boundary value problems  can be studied as an abstract Cauchy problem of a monotone operator on a Hilbert or Banach space  \cite{Ba,Br}; furthermore, under certain higher smoothness and stronger parabolicity conditions, such problems have been extensively studied in the theory of  quasilinear parabolic equation \cite{LSU,Li}.

Recently, for certain non-monotone functions $\sigma$, exact Lipschitz solutions  have been constructed for the initial-Neumann  problem  of (\ref{eq0}) in \cite{KY1,KY2,KY3}; in such cases, solutions can converge {\em weakly} to a function that is not a solution of (\ref{eq0}).

In this note,  assume that $\sigma\colon \R^n\to \R^n$  satisfies
\begin{equation}\label{gr0} 
  |\sigma(p)| \le c_1(|p| +1) \quad (p\in\R^n), 
\end{equation}
where $c_1>0$ is a constant, and we are interested in  {\em approximating sequences} of (\ref{eq0}),  especially when   $\sigma$ is non-monotone. 
Here, in general, by an approximating sequence  of (\ref{eq0}) we mean  a sequence $u^j$ in $L^2(\Omega_T)$ with $Du^j\in L^2(\Omega_T;\R^n)$ such that  
\begin{equation}\label{approx1}
 \lim_{j\to\infty} \|u^j_t-\dv\sigma(Du^j)\|_{H^{-1}(\Omega_T)}=0.
 \end{equation}
 
We attempt to characterize the limits of  approximating sequences  under the weak convergence in $H^1(\Omega_T)$.
Let
\begin{equation}\label{setA}
\Lambda=\{p\in\R^n\;|\; (\sigma(q)-\sigma(p))\cdot (q-p)\ge 0 \;\;\forall\,q\in\R^n\}.
\end{equation}
Then,  $\sigma$ is monotone if and only if  $\Lambda=\R^n.$ Note that $\Lambda$ may be empty; for example, if $\sigma$ is the function as shown in Figure \ref{fig2} below, then for function $-\sigma$ the set $\Lambda=\emptyset.$

\begin{lem}  For $p\in\R^n$, let   
\begin{equation}\label{set0}
\Gamma(p)=\{\beta\in \R^n\;|\; (\beta-\sigma(q))\cdot (p-q)\ge 0 \;\;\forall\, q\in \Lambda\}.
\end{equation}
{\em (If $\Lambda=\emptyset,$ then let $\Gamma(p)=\R^n$.)} Then $\Gamma(p)$ is a closed convex set in $\R^n$ containing $\sigma(p)$, with $\Gamma(p)=\{\sigma(p)\}$ if $p\in \Lambda^o,$ the interior of $\Lambda.$
\end{lem}
\begin{proof} Clearly $\Gamma(p)$ is closed, convex and contains $\sigma(p).$ Let $p\in \Lambda^o$ and $\beta\in \Gamma(p).$ Let $q=p-\epsilon \eta$, where $\epsilon>0$ and $\eta\in\R^n$. Then, for all $\epsilon>0$ sufficiently small, we have $q\in \Lambda$ and thus $(\beta-\sigma(q))\cdot (p-q)=(\beta-\sigma(p-\epsilon\eta))\cdot \epsilon \eta \ge 0;$ hence $(\beta-\sigma(p-\epsilon\eta))\cdot   \eta \ge 0$,  which by letting $\epsilon\to 0^+$  yields $(\beta-\sigma(p))\cdot   \eta \ge 0$. Since this inequality holds for all $\eta\in\R^n$, it follows  that $\beta=\sigma(p).$ Hence $\Gamma(p)=\{\sigma(p)\}$ if $p\in \Lambda^o.$
\end{proof}

In what follows, we denote by $g(p,\beta)$  the {\em convex hull} of function $|\sigma(p)-\beta|^2$ on $\R^n\times \R^n$. For $p\in\R^n$,  define
\begin{equation}\label{setZg}
Z(p)=\{\beta\in \R^n\;|\; g(p,\beta)=0\}, \quad \Sigma(p)=\Gamma(p)\cap Z(p).
\end{equation}
Then $Z(p)$ is a closed convex set in $\R^n$ containing $\sigma(p)$, and thus $\Sigma(p)$ is also a closed convex set in $\R^n$ containing $\sigma(p)$,  with $\Sigma(p)=\{\sigma(p)\}$ if $p\in \Lambda^o.$
 
 \begin{remk} (a) The set $\Sigma(p)$ can be very complicated, depending on the structure of the function $\sigma.$ In the one spatial dimension, two special functions of diffusion function $\sigma(p)$ are given as shown in Figures \ref{fig1} and \ref{fig2}.
 
 (b) Some interesting special structures in the set $\Sigma(p)$ can be characterized by a variational principle (see Proposition \ref{pro1}), which may have some relevance to the existence results in the recent work \cite{KY1,KY2,KY3}.
 \end{remk}
 
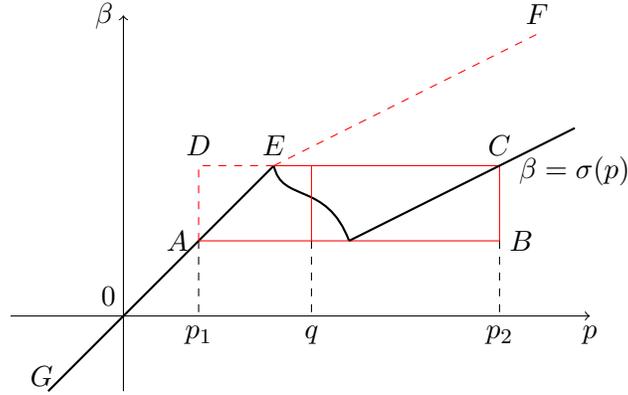
\begin{figure}[ht]
\begin{center}
\begin{tikzpicture}[scale =1]
    \draw[->] (-1.5,0) -- (6.2,0);
	\draw[->] (0,-1) -- (0,4);
 \draw[thick] (-1,-1)--(2,2);
 \draw[thick] (2,2) ..controls (2.1,1.5) and (2.7,1.8) ..(3,1);
  \draw[red] (2,2)--(5,2);
  \draw[dashed] (2.5,1)--(2.5,0);
      \draw[dashed] (1,1)--(1,0);
     \draw[dashed] (5,1)--(5,0);
      \draw[red] (5,2)--(5,1);
   \draw[red] (5,1) --(1,1);
     \draw[red,dashed] (2,2) --(1,2);
       \draw[red,dashed] (1,1) --(1,2);
   \draw[red,dashed] (2,2) --(5.5,3.75);
    \draw[thick] (3,1)--(6,2.5);
   \draw[red] (2.5,2)--(2.5,1);
    \draw (-0.2,0) node[above] {{$0$}};
    \draw (6, 1.6) node[above] {$\beta=\sigma(p)$};
 
 \draw (0, 4) node[left] {$\beta$}; 
   \draw (6.2, 0) node[below] {$p$}; 
        \draw (1, 0) node[below] {$p_1$};
          \draw (5, 0) node[below] {$p_2$};
                    \draw (2.5, 0) node[below] {$q$};
                    \draw(5,1) node[right]{$B$};
                    \draw(5,2) node[above]{$C$};
                    \draw(2,2) node[above]{$E$};             
                    \draw(1,1) node[left]{$A$};
                     \draw(-0.8,-0.8) node[left]{$G$};
                      \draw(1,2) node[above]{$D$};
                    \draw(5.5,3.75) node[above]{$F$};                 
\end{tikzpicture}
\end{center}
\caption{A special function $\sigma(p)$ in  1-D   (including the H\"ollig function \cite{Ho}). Here the set $\Lambda$ is $(-\infty,p_1]\cup [p_2,\infty).$ For all $q$ the set $Z(q)$ is the  vertical closed half-ray below the graph $GAEF.$ For $q\in [p_1,p_2],$ the set $\Gamma(q)$ is the vertical  cross-section of the closed rectangle $ABCD$, and the set $\Sigma(q)$ is the vertical  cross-section of the closed trapezoid $ABCE$.}
\label{fig1}
\end{figure}

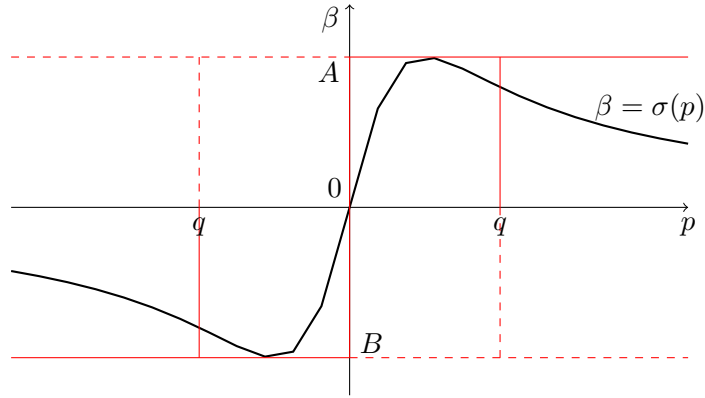
\begin{figure}[ht]
\begin{center}
\begin{tikzpicture}[scale =1]
    \draw[->] (-4.5,0) -- (4.5,0);
	\draw[->] (0,-2.5) -- (0,2.7);  
    \draw [thick,domain=-4.5:4.5] plot ({\x}, {4*(\x) /(1+(\x)*(\x))});   
    \draw (-0.2,0) node[above] {{$0$}};
    \draw (4, 1) node[above] {$\beta=\sigma(p)$};
    \draw[red,dashed] (-4.5,2) --(0,2);
    \draw[red,dashed] (0,-2) --(4.5,-2);
     \draw[red] (0,2) --(4.5,2);
    \draw[red] (-4.5,-2) --(0,-2);
     \draw[red] (0,2) --(0,-2);
    \draw[red] (2,0)--(2,2);
    \draw[red] (-2,0)--(-2,-2);
     \draw[red,dashed] (-2,0) --(-2,2);
     \draw[red,dashed] (2,0) --(2,-2);
    
 \draw (0, 2.5) node[left] {$\beta$}; 
   \draw (4.5, 0) node[below] {$p$};
                    \draw (2, 0) node[below] {$q$}; 
                     \draw (-2, 0) node[below] {$q$}; 
                    \draw(0,-1.8)  node[right]{$B$};             
                    \draw(0,1.8) node[left]{$A$};               
\end{tikzpicture}
\end{center}
\caption{Another special function $\sigma(p)$ in 1-D  with $\sigma(p)\to 0$ as $|p|\to \infty$  (including the Perona-Malik function \cite{PM}). Here the set $\Lambda$ is the point $\{0\}$. For all $q$ the set $Z(q)$ is the  vertical cross-section between two horizontal lines $\beta=A$ and $\beta=B$. The set $\Sigma(0)$ is the line segment $AB$  and, for $q\ne 0$, $\Sigma(q)$ is the  vertical line segment as shown in the first or third quadrant.}
\label{fig2}
\end{figure}

\begin{thm} \label{thm1}
 Let $u^j$ be an approximating sequence of $(\ref{eq0})$ in $H^1(\Omega_T)$ such that $u^j$ converges weakly to $ u^*$ in $H^1(\Omega_T).$ Then
\begin{equation}
\label{closure} 
u^*_t \in \dv \Sigma(Du^*)
\end{equation}
in the sense that $u^*_t=\dv\bar\sigma$ holds in $H^{-1}(\Omega_T)$  for some  $\bar\sigma\in L^2(\Omega_T;\R^n)$ with $\bar\sigma(x,t)\in \Sigma(Du^*(x,t))$ for almost every $(x,t)\in\Omega_T.$ In particular, if $\sigma$ is monotone, then $u^*$ is a solution of $(\ref{eq0})$.
\end{thm}

The weak convergence of $u^j$ is necessary in this result. In fact, we have the following result.

\begin{thm}\label{thm2} Let  $\bar u \in H^1(\Omega_T)$ be such that there exists a function  $ \bar v \in H^1(\Omega_T;\R^{n})$ satisfying  $\bar u=\dv \bar v$ and $\bar v_t\in Z(D\bar u)$ almost everywhere on $\Omega_T$. Then there exists an approximating sequence 
$u^j\in \bar u+ H_0^1(\Omega_T)$ such that    $u^j\to \bar u$ strongly in $L^2(\Omega_T).$
\end{thm}

Examples of the functions $\bar u$ and $\bar v=(\bar v_1,\dots,\bar v_n)$ in the theorem are given by
\[
\bar u=p_1x_1+\cdots + p_nx_n,\quad \bar v_i=\frac12 p_ix_i^2+\beta_i t \quad (i=1,2,\cdots,n),
\]
where $p=(p_1,\dots,p_n)$ and $\beta=(\beta_1,\dots,\beta_n)$ are such that $g(p,\beta)=0.$

The function $\bar u$ in Theorem \ref{thm2}  satisfies $\bar u_t\in \dv Z(D\bar u).$ In general,  the set $Z(p)$ is much larger than the set $\Sigma(p)$ (see Figures \ref{fig1} and \ref{fig2}).  Theorem \ref{thm2} holds even when $\sigma$ is monotone; therefore, even for parabolic equations, there are approximating sequences that  converge in the $L^2$-norm  but not in the weak $H^1$ convergence. 

An interesting problem is to study that for what functions $\bar u\in H^1(\Omega_T)$ one can find an exact  solution of (\ref{eq0}) in $\bar u+ H_0^1(\Omega_T);$ such a function $\bar u$ has been called a {\em subsolution} of  (\ref{eq0}) in the recent work  \cite{KY2,KY3}. From the constructions in \cite{KY2,KY3}, Proposition \ref{pro1} below suggests that a subsolution $\bar u$ should satisfy  a condition $\bar u_t\in \dv R(D\bar u)$ for a smaller set $R(p)$ contained in $\Sigma(p).$  
  
  \section{Compensated compactness: Proof of Theorem \ref{thm1}}
Let $N\ge 1$ and $G$ be a bounded domain in $\R^N.$ Given a vector function $U=(U_1,\dots,U_N)\in L^2(G;\R^N),$ let
\[
\Dv U=\sum_{k=1}^N \frac{\partial U_k}{\partial y_k},\quad \Curl U=(\frac{\partial U_k}{\partial y_l}-\frac{\partial U_l}{\partial y_k})_{1\le k,l\le N}.
\]
Then $\Dv U\in H^{-1}(G)$ and $\Curl U\in H^{-1}(G;\M^{N\times N})$ are  well-defined, where $\M^{N\times N}$ is the space of all $N\times N$ matrices.

We need the following compensated compactness result known as the div-curl lemma \cite{T}.

\begin{lem}[div-curl lemma] Let $V^j, \,W^j\in L^2(G;\R^N)$ satisfy  $V^j\wcon V^*, \,W^j\wcon W^*$ weakly in $L^2(G;\R^N)$.  Assume $\Dv V^j$  and $\Curl W^j$ converge  strongly in $H^{-1}(G)$ and $H^{-1}(G;\M^{N\times N}),$ respectively. Then $V^j\cdot W^j$ converges to $V^*\cdot W^*$ in the sense of distributions on $G;$ that is, for all $\phi\in C^\infty_0(G),$
\[
\lim_{j\to\infty}\int_G \phi (y) V^j(y)\cdot W^j(y) \,dy = \int_G\phi (y) V^*(y) \cdot W^*(y)\,dy.
\]
\end{lem}
We refer to \cite{T} for proof and  to \cite{CDM} for some recent developments on the div-curl lemma.

\subsection*{Proof of Theorem \ref{thm1}}
Let $u^j$ be an approximating sequence of $(\ref{eq0})$ in $H^1(\Omega_T)$ such that $u^j$ converges weakly to $ u^*$ in $H^1(\Omega_T).$ By selecting a subsequence if necessary, we assume that $u^j\to u^*$ strongly in $L^2(\Omega_T)$ and $\sigma(Du^j)\wcon \bar \sigma$ weakly in $L^2(\Omega_T;\R^n).$ 
 
 Clearly, $u^*_t=\dv \bar\sigma$ in $H^{-1}(\Omega_T).$ 
 
   As above, let $g(p,\beta)$ be the convex hull of $|\sigma(p)-\beta|^2.$ Then
\begin{equation}\label{g-growth}
0\le g(p,\beta)\le c_1(|p|^2+|\beta|^2+1)\quad \forall\;(p,\beta)\in \R^n\times \R^n.
\end{equation}
 Since   $g(Du^j,\sigma(Du^j))=0$, by convexity and (\ref{g-growth}),  we have 
\[
0\le \int_{\Omega_T} g(Du^*(x,t),\bar\sigma(x,t))\,dxdt\le \liminf_{j\to\infty}\int_{\Omega_T} g(Du^j,\sigma(Du^j))=0;
\]
hence,  $g(Du^*(x,t),\bar\sigma(x,t))=0$ and 
\[
\bar\sigma(x,t)\in Z(Du^*(x,t)) \quad a.e.\;\; \Omega_T.
\]
Let $V^j=(\sigma(Du^j),-u^j)$ and $W^j=(Du^j,u^j_t)$ be  functions in $L^2(\Omega_T;\R^{n+1}).$ 
Then $V^j\wcon V^*=(\bar\sigma,-u^*)$ and $W^j\wcon  W^*=(Du^*,u^*_t)$ both weakly in $L^2(\Omega_T;\R^{n+1})$. 
Moreover,
$\Dv V^j=\dv \sigma(Du^j)-u^j_t\to 0$ strongly in $H^{-1}(\Omega_T)$  and $\Curl W^j=0$  in $H^{-1}(\Omega_T;\M^{(n+1)\times (n+1)}).$ Hence, by the div-curl lemma,   it follows that
$
V^j\cdot W^j \wcon V^*\cdot W^*$ in the sense of distributions on $\Omega_T$; that is,
 \[
 \lim_{j\to\infty}\int_{\Omega_T} \phi V^j\cdot W^j \,dxdt = \int_{\Omega_T} \phi V^*\cdot W^*\,dxdt\quad \forall \,  \phi\in C^\infty_0(\Omega_T).
 \]
Since  $V^j\cdot W^j =\sigma(Du^j)\cdot Du^j-u^ju^j_t$ and $V^*\cdot W^*=\bar\sigma \cdot Du^* -u^*u^*_t,$   we thus obtain
\begin{equation}\label{eq00}
 \lim_{j\to\infty}\int_{\Omega_T} \phi \sigma(Du^j)\cdot Du^j \,dxdt = \int_{\Omega_T} \phi \bar\sigma \cdot Du^*\,dxdt\quad \forall \,  \phi\in C^\infty_0(\Omega_T).
\end{equation}

From (\ref{eq00}), we have, for all $q\in\R^n$ and $\phi\in C_0^\infty(\Omega_T)$,
\[
 \lim_{j\to\infty}\int_{\Omega_T} \phi (\sigma(Du^j)-\sigma(q))\cdot (Du^j -q)\,dxdt = \int_{\Omega_T} \phi (\bar\sigma-\sigma(q)) \cdot (Du^*-q)\,dxdt.
\]
Let $q\in \Lambda$; then, for all $\phi\in C^\infty_0(\Omega_T)$ with $\phi\ge 0$, we have
\[
\int_{\Omega_T} \phi (\bar\sigma-\sigma(q)) \cdot (Du^*-q)\,dxdt \ge 0.
\]
This implies that $(\bar\sigma-\sigma(q)) \cdot (Du^*-q)\ge 0$ almost everywhere in $\Omega_T$, which is true for all $q\in\Lambda$, and hence  we have
\[
\bar\sigma(x,t) \in \Gamma(Du^*(x,t)) \quad a.e. \,\;  \Omega_T.
\]

Therefore, $\bar\sigma(x,t)\in \Sigma(Du^*(x,t))$  almost everywhere in $\Omega_T$ and $u^*_t=\dv \bar\sigma$ in $H^{-1}(\Omega_T)$. This completes the proof.

\section{Associated first-order system and  variational problem: Proof of Theorem \ref{thm2}} 
 
 Let $u\in L^2(\Omega_T)$ and $v \in L^2(\Omega_T;\R^{n})$ be  such that $Du,v_t,  \dv v$ are of  $L^2(\Omega_T).$ For such $(u,v)$, we introduce  the functional
\[
I(u,v)=\|v_t-\sigma(Du) \|^2_{L^2(\Omega_T)} + \|u-\dv v\|^2_{L^2(\Omega_T)}.
\]
If $(u,v)$ solves the  first-order differential system
 \begin{equation}
\label{sys0} 
u=\dv v,\;\; v_t=\sigma(Du) \quad \mbox{on $\; \Omega_T$}
\end{equation} 
  in the sense of distribution, then $u$ is a  solution of (\ref{eq0}) in the same sense.  
 
\begin{lem}\label{lem3} Let $w=(u,v)\in L^2(\Omega_T;\R^{1+n})$ be  such that $Du,v_t,  \dv v$ are of  $L^2(\Omega_T).$ Then
\[
\|u_t-\dv \sigma(Du)\|_{H^{-1}(\Omega_T)} \le  \sqrt{I(u,v)}. 
\]
Therefore, if a sequence $(u^j,v^j)$ of such functions satisfies $I(u^j,v^j)\to 0$, then   $u^j$ is an approximating sequence $u^j$ of $(\ref{eq0})$. 
\end{lem}

\begin{proof} If $\phi\in C^\infty_0(\Omega_T)$, then 
\[
\int_{\Omega_T} \phi_t \dv v  \,dxdt=\int_{\Omega_T} v_t\cdot D\phi \,dxdt.
\]
By density, this equality also holds if $\phi\in H^1_0(\Omega_T).$ Hence, for all $\phi\in H^1_0(\Omega_T)$, we have 
\[
\langle u_t-\dv \sigma(Du),\phi\rangle =-\int_{\Omega_T} u\phi_t \,dxdt +\int_{\Omega_T} \sigma(Du)\cdot D\phi\,dxdt
\]
\[
=-\int_{\Omega_T} (u-\dv v) \phi_t\,dxdt + \int_{\Omega_T} (\sigma(Du)-v_t)\cdot D\phi\,dxdt 
\]
\[
\le  \|u-\dv v\|_{L^2(\Omega_T)}\, \|\phi_t\|_{L^2(\Omega_T)}+ \|v_t-\sigma(Du) \|_{L^2(\Omega_T)}\, \|D\phi\|_{L^2(\Omega_T)}\]
\[
\le \sqrt{I(u,v)}  \, (\|\phi_t\|^2_{L^2(\Omega_T)} +\|D\phi\|^2_{L^2(\Omega_T)} )^{1/2} \le   \sqrt{I(u,v)} \,\|\phi\|_{H^1_0(\Omega_T)}.
\] 
Therefore,
\[
\|u_t-\dv \sigma(Du)\|_{H^{-1}(\Omega_T)}=\sup_{\|\phi\|_{H^1_0(\Omega_T)}\le 1} \langle u_t-\dv \sigma(Du),\phi\rangle \le  \sqrt{I(u,v)}. 
\]
\end{proof}

We now put the functional $I(u,v)$ in the context of the calculus of variations. 

In the following, for a vector function $w=(u,v)\in W^{1,1}_{loc}(\Omega_T;\R^{1+n})$, where $u\colon \Omega_T\to\R$ and $v\colon\Omega_T\to \R^n$,  we write its  gradient $\nabla w(x,t)$ as a matrix
\[
\nabla w(x,t) = \begin{pmatrix}
Du(x,t)&u_t(x,t)\\Dv(x,t)&v_t(x,t)
\end{pmatrix},
\]
where $Du$ is considered as a $1\times n$ matrix, $Dv$ as an $n\times n$ matrix, and $v_t$ as an $n\times 1$ matrix.  

Accordingly, we denote by $\MN$ the space of $(1+n)\times (n+1)$ matrices $\xi$  written as 
$\xi=\begin{pmatrix}
p&r\\B&\beta
\end{pmatrix}$,
where  $p$ is a $1\times n$ matrix  viewed as a vector in $\R^n$,  $\beta$ is an $n\times  1$ matrix  also viewed  as a vector in $\R^n$, $r\in\R$ is a scalar, and $B\in \M^{n\times n}$ is an $n\times n$ matrix.
                 
Then, for $w=(u,v)\in H^1(\Omega_T;\R^{1+n})$,  the functional $I(u,v)$ defined above can be written as an integral  functional 
\begin{equation}\label{fun-I}
I(u,v)=I(w)=\int_{\Omega_T} f(w(x,t),\nabla w(x,t))dxdt
\end{equation}
with  the density function  $f\colon \R^{1+n}\times \MN\to \R$ given by
\begin{equation}\label{fun-f}
f(\omega,\xi)=|\sigma(p)-\beta|^2 + (s-\tr B)^2,
\end{equation}
where $\omega=(s,z)\in\R^{1+n}$ and $\xi=\begin{pmatrix}
p&r\\B&\beta
\end{pmatrix}\in \MN.$  

Clearly, $f$ satisfies the growth condition
\begin{equation}\label{f-growth}
0\le f(\omega,\xi)\le C(|\omega|^2 +|\xi|^2+1) 
\end{equation}
on $ \R^{1+n}\times \MN$. Define
\begin{equation}\label{sharp-f}
f^\#(\omega,\xi)=\inf_{\zeta\in  C^\infty_0(Q;\R^{1+n})} \frac{1}{|Q|}\int_Q f(\omega,\xi+\nabla\zeta(x,t))dxdt;
\end{equation}
here $Q\subset \R^{n+1}$ is a cube. 

Then, for each $\omega\in\R^{1+n}$, $f^\#(\omega,\xi)$ is independent of $Q$ and  quasiconvex in $\xi$ in the sense of Morrey \cite{M} (see also \cite{AF,B,D}).

\begin{pro}\label{lem2}  One has that $f^\#(\omega,\xi)=g(p,\beta)+(s-\tr B)^2,$ where  $g(p,\beta)$ is the convex hull of $|\sigma(p)-\beta|^2$ on $\R^n\times \R^n.$
\end{pro}

\begin{proof} Given  $\zeta=(\phi,\psi)\in  C^\infty_0(Q;\R^{1+n}),$  we have
\[
f(\omega,\xi+\nabla\zeta) = |\beta+\psi_t-\sigma(p+D\phi)|^2 +(s-\tr B)^2-2(s-\tr B) (\dv \psi) + (\dv \psi)^2,
\]
and hence $f^\#(\omega,\xi)=\tilde g(p,\beta)+(s-\tr B)^2,$ where
\[
\tilde g(p,\beta)=  \inf_{(\phi,\psi)\in  C^\infty_0(Q;\R^{1+n})} \frac{1}{|Q|}\int_Q  \Big (|\beta+\psi_t-\sigma(p+D\phi)|^2 +(\dv \psi)^2\Big ) dxdt.
\]
Clearly,  $\tilde g(p,\beta)\le | \sigma(p)-\beta|^2.$   

Let $\rho_\epsilon(\xi)=\tilde \rho_\epsilon(|\xi|)$ be a standard mollifying sequence on $\MN$ and let 
$f_\epsilon^\#(\omega,\cdot) =f^\#(\omega,\cdot)  * \rho_\epsilon(\cdot)$ be the smoothing sequence of $f^\#.$   Since $f^\#(\omega,\xi)$ is quasiconvex (and thus continuous) in $\xi$, it follows that  $f_\epsilon^\#(\omega,\xi)$ is quasiconvex and smooth in $\xi$ for each $\epsilon>0.$ Write $f_\epsilon^\#=g_\epsilon + h_\epsilon$, where  
\[
g_\epsilon =\tilde g * \rho_\epsilon=g_\epsilon(p,\beta),\quad h_\epsilon  =(s-\tr B)^2 * \rho_\epsilon =(s-\tr B)^2 + c_\epsilon \]
 with $c_\epsilon$ being a constant such that $c_\epsilon\to 0$ as $\epsilon\to 0^+.$

From the quasiconvexity of $f_\epsilon^\#(\omega,\xi)$ in $\xi$, it follows that $f_\epsilon^\#(\omega,\xi)$ is rank-one convex in $\xi$; see \cite{B,D}. Therefore, for each rank-one matrix $\eta$ in $\MN$ of the form
\[
\eta=\begin{pmatrix}
1\\\alpha
\end{pmatrix} \otimes \begin{pmatrix}
q&1
\end{pmatrix}=\begin{pmatrix}
q&1\\\alpha\otimes q & \alpha
\end{pmatrix} \quad (q\in\R^n,\; \alpha\in \R^n),
\]
 the function $l(t)=f_\epsilon^\#(\omega,\xi+t\eta)=g_\epsilon(p+tq,\beta+t\alpha)+(s-\tr B-t\alpha\cdot q)^2 +c_\epsilon$  is a convex function of $t$ on $\R.$  Hence
\[
l''(0)=[(g_\epsilon)_{pp}(p,\beta) q]\cdot q + 2 [(g_\epsilon)_{p\beta}(p,\beta) \alpha]\cdot q + [(g_\epsilon)_{\beta\beta}(p,\beta) \alpha] \cdot \alpha + 2 (\alpha\cdot q)^2\ge 0.
\]
In this inequality, replace $(q,\alpha)$ by $(kq,k\alpha)$ with $k>0$, cancel $k^2$ and  let $k\to 0.$ Then it follows that
\[
[(g_\epsilon)_{pp}(p,\beta) q]\cdot q + 2 [(g_\epsilon)_{p\beta}(p,\beta) \alpha]\cdot q + [(g_\epsilon)_{\beta\beta}(p,\beta) \alpha] \cdot \alpha\ge 0
\]
for all $p,\beta, q,\alpha \in\R^n$. This proves that $g_\epsilon$ is convex on $\R^n\times \R^n$ for each $\epsilon>0;$ hence, $\tilde g$  is convex on $\R^n\times \R^n$. 

Finally, if $h(p,\beta)$ is a convex function satisfying  $h(p,\beta)\le | \sigma(p)-\beta|^2$ for all $p,\beta \in \R^n,$ then, by Jensen's inequality, for all $(\phi,\psi)\in C^\infty_0(Q;\R^{1+n}),$  
\[
h(p,\beta)\le \frac{1}{|Q|}\int_Q h(p+D\phi,\beta+\psi_t)dx dt 
\le \frac{1}{|Q|}\int_Q  |\beta+\psi_t-\sigma(p+D\phi)|^2  dxdt 
\]
\[
\le \frac{1}{|Q|}\int_Q  (|\beta+\psi_t-\sigma(p+D\phi)|^2 +(\dv \psi)^2) dxdt.
\]
Hence $h(p,\beta)\le \tilde g(p,\beta);$ therefore, $\tilde g(p,\beta)$ is the largest convex function below the function  $| \sigma(p)-\beta|^2,$ and thus  $\tilde g=g.$
\end{proof}

\begin{remk}  Given $\omega\in\R^{1+n}$, let $Rf(\omega,\cdot)$ be the {\em rank-one convex hull} of $f(\omega,\cdot)$ on $\MN$  (see \cite{D}); then, it can be shown that
\[
Rf(\omega,\xi)=h(p,\beta)+(s-\tr B)^2
\]
 for some function $h(p,\beta)\ge g(p,\beta).$ Similarly as  in the proof of Proposition \ref{lem2}, it follows  that $h$ is convex, and hence $h=g$ and 
\[
Rf(\omega,\xi)=f^\#(\omega,\xi)=g(p,\beta)+(s-\tr B)^2.
\] 
Therefore, the rank-one convex hull of $f$ does not produce more special features than the quasiconvex hull of $f$.
\end{remk}

The following result shows that some special structures in the set $\Sigma(p)$ defined above  have a variational description; this result also shows that the subsolutions $\bar u$ used in \cite{KY2,KY3}  automatically satisfy $\bar u_t\in \dv \Sigma(D\bar u)$.

\begin{pro}\label{pro1}
For  $\lambda>0,$ let
\begin{equation}\label{g-lambda}
g_\lambda(p,\beta) =  \inf_{{(\phi,\psi)\in  C^\infty_0(Q;\R^{1+n})}\atop { \|D\phi\|_{L^2(Q)}< \lambda, \; \dv\psi=0}} \frac{1}{|Q|}\int_Q    |\beta+\psi_t-\sigma(p+D\phi)|^2 \, dxdt.
\end{equation}
Then, $\beta\in\Sigma(p)$ if $g_\lambda(p,\beta)=0.$
\end{pro}
 \begin{proof} Assume $g_\lambda(p,\beta)=0$. From the formula of $g=\tilde g$ in the proof of Proposition \ref{lem2}, it follows  that $0\le g\le g_\lambda;$ hence $\beta\in Z(p).$ To prove $\beta\in \Gamma(p),$ let $(\phi^j,\psi^j)\in  C^\infty_0(Q;\R^{1+n})$ be such that $\|D\phi^j\|_{L^2(Q)}< \lambda, \; \dv\psi^j=0,$ and
 \[
\lim_{j\to\infty} \int_Q    |\beta+\psi^j_t-\sigma(p+D\phi^j)|^2 \, dxdt =0.
\]
Let $\eta^j=\beta+\psi^j_t-\sigma(p+D\phi^j).$ Then $\eta^j\to 0$ in $L^2(Q).$ Let $q\in \Lambda.$ Then
$(\sigma(p+D\phi^j)-\sigma(q))\cdot (p+D\phi^j-q)\ge 0;$ hence
\begin{equation}\label{ineq-7}
\int_Q (\beta+\psi^j_t -\eta^j-\sigma(q))\cdot (p+D\phi^j-q)\,dxdt\ge 0.
\end{equation}
Since 
\[
\int_Q \psi^j_t\cdot D\phi^j\,dxdt=\int_Q \phi^j_t\dv\psi^j\,dxdt=0,
\]
  $\eta^j\to 0$ in $L^2(Q)$ and $\|D\phi^j\|_{L^2(Q)}<\lambda$, in (\ref{ineq-7}) letting $j\to\infty$, it follows that
  \[
  (\beta-\sigma(q))\cdot (p-q)\ge 0.
  \]
  Hence $\beta\in \Gamma(p);$ this completes the proof.
 \end{proof}

 \subsection*{Proof of Theorem \ref{thm2}} 
Let $\bar w=(\bar u,\bar v)\in H^1(\Omega_T;\R^{1+n})$ satisfy  $\bar u=\dv \bar v$ and $\bar v_t\in Z(D\bar u)$ almost everywhere on $\Omega_T$. Consider the {\em relaxed energy} (see \cite{AF,D})
\[
I^\#(w) =\int_{\Omega_T} f^\#(w(x,t),\nabla w(x,t))dxdt 
\]
for all $w\in H^1(\Omega_T;\R^{1+n}).$ Then $I^\#(\bar w)=0.$ By \cite[Theorem 9.8]{D},  there exists a sequence 
$w^j\in \bar w+ H_0^1(\Omega_T;\R^{1+n})$ such that  $I(w^j)\to 0$ and $w^j\to \bar w$ strongly in $L^2(\Omega_T;\R^{1+n}).$ If $w^j=(u^j,v^j)$ then, by Lemma \ref{lem3}, $u^j\in  \bar u+ H_0^1(\Omega_T)$ is an approximating sequence of (\ref{eq0}) such that $u^j\to \bar u$ in $L^2(\Omega_T).$

This completes the proof.

\end{document}